\documentclass[12pt,a4paper]{amsart}
\usepackage{preamble}

\title[No-propagation for random linear response] 
{
No-propagate algorithm for linear responses of random chaotic systems 
}

\begin{document}

\begin{abstract}

We develop the no-propagate algorithm for sampling the linear response of random dynamical systems, which are non-uniform hyperbolic deterministic systems perturbed by noise with smooth density.
We first derive a Monte-Carlo type formula and then the algorithm, which is different from the ensemble (stochastic gradient) algorithms, finite-element algorithms, and fast-response algorithms; it does not involve the propagation of vectors or covectors, and only the density of the noise is differentiated, so the formula is not cursed by gradient explosion, dimensionality, or non-hyperbolicity.
We demonstrate our algorithm on a tent map perturbed by noise and a chaotic neural network with 51 layers $\times$ 9 neurons.

By itself, this algorithm approximates the linear response of non-hyperbolic deterministic systems, with an additional error proportional to the noise.
We also discuss the potential of using this algorithm as a part of a bigger algorithm with smaller error.


\end{abstract}

\maketitle

\section{Introduction}
\label{s:intro}

\subsection{Literature review}
\hfill\vspace{0.1in}

The long-time-average statistic of a chaotic system is given by the physical measure, which is the invariant measure obtained as the limit of evolving a Lebesgue measure  \cite{young2002srb,srbmap,srbflow}.
We are interested in the derivative of the long-time-average of an observable function with respect to the parameters of the system.
This derivative is also called the linear response, and it is one of the most basic numerical tools for many purposes.
For example, in aerospace design, we want to know the how small changes in the geometry would affect the average lift of an aircraft, 
which was answered for non-chaotic fluids \cite{Jameson1988} but only partially answered for not-so-chaotic fluids \cite{Ni_CLV_cylinder}.
In climate change, we want to ask what is the temperature change caused by a small change of $CO_2$ level \cite{Ghil2020,Froyland2021}.
In machine learning, we want to extend the conventional backpropagation method to cases with gradient explosion; current practices have to avoid gradient explosion but that is not always achievable \cite{clip_gradients2,resnet}.
These are all linear responses, and an efficient algorithm would help the advance of the above fields.

The are two major settings for computing linear responses: deterministic and random.
For deterministic systems, both the proof and the algorithm for linear responses require enough amount of hyperbolicity.
The most well-known formulas are the ensemble formula and the operator formula , which have been rigorously proved under various hyperbolicity assumptions, and they are formally equivalent under integration-by-parts \cite{Gallavotti1996,Ruelle_diff_maps,Ruelle_diff_maps_erratum,Ruelle_diff_flow,Dolgopyat2004,Gouezel2006,Baladi2007,Korepanov2016}.
Since most real-life systems have large degrees of freedom, the dynamical systems are high-dimensional, where the most efficient way to sample physical-measure related and integrated quantities is Monte-Carlo, that is, to take long-time average on a sample orbit.
The essential difficulties for such Monte-Carlo algorithm for linear responses are the explosion phenomena in the pushforward on vectors, and the roughness in the unstable/stable subspaces.
These difficulties are efficiently solved by the fast response formulas and algorithms with cost $O(Mu)$ per step, where $M$ is the system dimension and $u$ is the unstable dimension \cite{fr,Ni_asl,Ni_nilsas,Ni_NILSS_JCP,Ni_CLV_cylinder}.
The fast formulas can be interpreted as recursive formulas for the unstable perturbation of measure transfer operator on unstable manifolds \cite{far,TrsfOprt}.
However, although the above works tolerate non(uniform)-hyperbolicity to some degree, they all fail beyond certain amount of non-hyperbolicity \cite{Baladi2007,Gottwald2016,Wormell2019,wormell22,wormell_pieceMap}.

For random systems, the linear response formula can take another form \cite{Hairer2010,bahsoun20,Antown2022}, and annealed linear response can exist when quenched linear response does not \cite{sedro23}.
Here annealed means to average over the probability space whereas quenched means to fix a specific sequence of (previously random) maps.
In particular, when we have a deterministic maps perturbed by independent noise with smooth density, we can arrange the derivatives so that they only hit the density function of the noise.
Hence, the formula works regardless of whether the deterministic part is hyperbolic or not.
So far the numerical algorithms are essentially finite-element method, which incurs an approximation error, and can provide more information than our approach, such as the second eigenvalue of the transfer operator, which is useful for improving mixing rate \cite{Antown2022}.
However, finite-element method is not suitable for high-dimensions.

\subsection{Main results}
\hfill\vspace{0.1in}

In this paper we develop a Monte-Carlo method for sampling linear responses of random dynamical systems with smooth density.
This algorithm runs on sample orbits; this incurs an additional sampling error, but it does not have the finite-element approximation error.
Overall, compared to the previous finite-element method for random linear responses, this orbit-based Monte-Carlo approach is more suitable for high-dimensions.

The first part of our paper derives an integral formula for the linear response of deterministic systems perturbed by noise.
The main points of interests are
\begin{itemize}
  \item All dummy variables are from probabilities which we can easily sample.
  \item The formula does not involve Jacobian matrices hence does not incur propagations of vectors or covectors (parameter gradients).
\end{itemize}
Such a formula enables Monte-Carlo sampling of the linear response; in the context of dynamics, the only interpretation of Monte-Carlo should be to sample by a few orbits.
Also, such a formula is not hindered by most undesirable features of the Jacobian of the deterministic part, such as singularity or non-hyperbolicity.

Our main theorem is about the random dynamical system $X_{n+1}=f(X_n)+Y_{n+1}$, where $h_n$ and $p_n$ are the probability density for $X_n$ and $Y_n$, and we know $h_0$ and all $p_n$'s.
The perturbation $\delta(\cdot):=\partial (\cdot)/ \partial\gamma|_{\gamma=0}$.
$\Phi$ is a fixed $C^2$ observable function.
Then we prove the following formula, which is (closely related to) a long-known wisdom. 
But we find it a novel usage, that is, we can Monte-Carlo sample it by fewest samples; in particular, one orbit can provide information for many different decorrelation step $m$'s.

\begin{restatable} [No-propagate formula for $\delta h_T$]{theorem}{goldbach}
\label{t:finiteT}
Under the assumptions of \cref{l:niu},
\[ \begin{split}
  \delta \int \Phi(x_T) h_T(x_T) dx_T
  = - \int \Phi(x_T) \sum_{m=0} ^ {T-1} \left( \delta f_\gamma x_{m} \cdot \frac {dp}{p} (y_{m+1}) \right) h_0(x_0) dx_0 (pdy)_{1\sim T}.
\end{split} \]  
Here $dp$ is the differential of $p$, $(pdy)_{1\sim T}:=p(y_1)dy_1\cdots p(y_T)dy_T$ is the independent distribution of $Y_n$'s, and $x_m$ is a function of dummy variables $x_0, y_1, \cdots, y_m$.
\end{restatable}

Then we extend this basic theorem to several related cases, which are implemented by the numerical algorithms in this paper.
\Cref{l:dhTn} is for the time-inhomogeneous case, where the phase space is different for different time steps.
In \cref{l:dh} is for the long-time case, so we no longer need to sample the initial probability $h_0$.

\Cref{s:intui} gives a pictorial intuition for the main lemmas.
With this intuition we can see, without being hindered by the complexity from notations, how to extend to some other cases, which shall be useful in later papers.
\Cref{s:generalize} extends to cases where the additive noise depends on locations and parameters; we also generalize to random dynamical systems with smooth densities.

The second part of our paper is about numerical realizations of our formulas.
Section~\ref{s:algorithm} gives a detailed list for the algorithm.
Section \ref{s:examples} illustrates the algorithm on several numerical examples.
In particular, we run on an example where the deterministic part does not have a linear response, but adding noise and using the no-propagate algorithm give a reasonable reflection of the observable-parameter trend.
We shall also run on a unstable neural network with 51 layers $\times$ 9 neurons.

Section~\ref{s:discuss} discusses the relation to and how to transit to deterministic cases.
In particular, we propose a program which combines the three famous formulas for linear responses, i.e. the ensemble formula, operator formula, and the random formula, to obtain a perhaps universal algorithm to compute the best approximation of linear responses.
The main idea is to add local but large-in-amplitude noise only to non-hyperbolic regions.
This trinity still requires non-trivial techniques, which we shall defer to later papers.
This current paper focus more on the random formula, whereas the fusion of the ensemble and operator formula has been done in our fast response algorithms \cite{fr,far,TrsfOprt}.
\Cref{s:kernel} compares no-propagate algorithm with some seemingly similar algorithms.

\section{Deriving the no-propagate linear response formula}
\label{s:derive}

\subsection{Preparation}
\hfill\vspace{0.1in}
\label{s:prep}

Let $f_\gamma$ be a family of $C^2$ map on the Euclidean space $\R^M$ of dimension $M$ parameterized by $\gamma$.
Assume that $\gamma \mapsto \tf_\gamma$ is $C^1$ from $\R$ to the space of $C^2$ maps.
The random dynamical system in this paper is given by 
\begin{equation} \begin{split} \label{e:dynamic}
  X_n =  f_\gamma (X_{n-1}) + Y_n ,
  \quad \textnormal{where} \quad 
  Y_n \iid  p.
\end{split} \end{equation}
The default value is $\gamma=0$, and we denote $f:=f_{\gamma=0}$.
Here $p$ is any fixed $C^2$ probability density function.
For convenience, here we arrange $f$ before adding noise, but the result is equivalent should we write the expression of our dynamics as 
or $ Z_n = f_\gamma (Z_{n-1}+Y_{n-1}) $.

Moreover, except for \cref{l:dh}, our results also apply to time-inhomogeneous cases, where $f$ and $p$ are different for each step.
More specifically, the dynamic is given by
\begin{equation} \begin{split} \label{e:dynamicinhomo}
  X_n =  f_{\gamma,n-1}(X_{n-1}) + Y_n ,
  \quad \textnormal{where} \quad 
  Y_n \sim  p_n.
\end{split} \end{equation}
Here $X_n\in \R^{M_n}$ are not necessarily of the same dimension.
We shall exhibit the time dependence in \cref{l:dhTn} and \cref{s:dhT}.
On the other hand, for the infinite time case in \cref{l:dh}, we want to sample by only one orbit; this requires that $f$ and $p$ be repetitive among steps.

We define $h_T$ as the density of pushing-forward the initial measure $h_0$ for $T$ steps.
\[ \begin{split}
  h_T:=  ( L_p L_{f_\gamma})^T h_0.
\end{split} \]
$h_T=h_{T,\gamma}$ depends on $\gamma$ and also $h_0$.
Here $L_f$ is the measure transfer operator of $f$,
which are defined by the integral equality
\begin{equation} \begin{split} \label{e:Lh}
  \int \Phi(x) \, (L_f h)(x) dx
  := 
  \int \Phi(fx)  h(x) dx.
\end{split} \end{equation}
Here $\Phi\in C^2(\R^M)$ is any fixed observable function.
When $f$ is bijective, there is an equivalent pointwise definition of $L_f h$,
$
  (L_f h) (x) := \frac h { |Df| }(f^{-1}x),
$
but this paper does not use this pointwise definition.

On the other hand, $L_p h'$ is pointwisely defined by convolution with density $p$:
\begin{equation} \begin{split} \label{e:Lppt}
  L_p h' (x) 
  = \int h'(x-y) p(y) dy
  = \int h'(z) p(x-z) dz.
\end{split} \end{equation}
We shall also use the integral equality
\begin{equation} \begin{split} \label{e:Lpint}
  \int \Phi (x) L_p h' (x) dx
  = \int p(y) \left( \int \Phi(x) h'(x-y)  dx \right) dy
  \,\overset{z=x-y}{=}\,
  \iint \Phi(y+z) h'(z) p(y) dz dy.
\end{split} \end{equation}
If we want to compute the above integration by Monte-Carlo method, then we should generate i.i.d. $Y_l=y_l$ and $Z_l=z_l$ according to density $p$ and $h'$, then compute the average of $\Phi(y_l+z_l)$.

The linear response formula for finite time $T$ is an expression of $\delta h_{\gamma,T}$ by $\delta f_\gamma$, and 
\[ \begin{split}
  \delta(\cdot)
  :=\left. \frac {\partial (\cdot)}{\partial \gamma} \right|_{\gamma=0}.
\end{split} \]
Here $\delta$ may as well be regarded as small perturbations.
Note that by our notation,
\[ \begin{split}
  \delta f_\gamma(x) \in T_{fx} \R^M;
\end{split} \]
that is, $\delta f_\gamma(x)$ is a vector at $fx$.
The linear response formula for finite-time is given by the Leibniz rule,
\begin{equation} \begin{split} \label{e:xu}
  \delta h_T
  = \delta (L_p L_{f_\gamma} )^T h_0
  = \sum_{n=0}^{T-1} ( L_p L_f )^{n} \, \delta(L_p L_{f_\gamma} ) \, (L_p L_f )^{T-1-n} h_0.
\end{split} \end{equation}
We shall give an expression of the perturbative transfer operator $\delta (L_p  L_{f_\gamma})$ later.

We make two assumptions for the case where $T\rightarrow\infty$.
First, we assume the existence of the density $h$ of the (ergodic) physical measure, which is the unique limit of $h_T$ in the weak$*$ topology:
\[ \begin{split}
  \lim_{T\in\N,\, T\rightarrow \infty} (L_p L_{f_\gamma})^T h_{0} = h,
\end{split} \]
note that $h$ depends on $\gamma$ but not on $h_0$ if $h_0$ smooth.
We also assume that the linear response for physical measure exists, that is, we can substitute the limit $ (L_p L_f )^{T-n} h_0 = h$ into \cref{e:xu} to get:
\begin{equation} \begin{split} \label{e:lininf}
  \delta h
  = \sum_{n\ge0} (L_f L_p)^{n} \delta(L_p L_{f_\gamma} ) h.
\end{split} \end{equation}

Our assumptions on the unique existence of $h$ and $\delta h$ can be proved under various assumptions; the formula is the same regardless of the assumptions.
For us, the most prominent assumptions are:
(1) the noise has smooth density;
(2) topological transitivity;
(3) the compactness of the phase space.
The last assumption may be replaced by assuming the existence of compact attractors which confines/concentrates the random walk.
These assumptions are typically much more forgiving than deterministic cases: in particular, we do not need hyperbolicity from $f$, which is typically untrue for real-life systems.

This paper does not attempt to prove the existence of $h$ and $\delta h$;
rather, we seek to compute these on one or a few orbits.

\subsection{Monte Carlo expression for measure perturbations}
\hfill\vspace{0.1in}
\label{s:MCexpression}

First, we write out the expression of $(L_p L_{f_\gamma})^n h_0$ and $\delta  L_p L_{f_\gamma}h_0$.
Then we use these to get the Monte-Carlo expression of linear responses.

\begin{lemma} \label{l:PhiMC}
For any $C^2$ (not necessarily positive) densities $h_0$ and $p$, any $C^2$ map $f$, any $n\in \N$,
  \[ \begin{split}
    \int \Phi(x_n) ((L_pL_f)^n h_0) (x_n) dx_n
  = \int \Phi(x_n) h_0(x_0) dx_0 p(y_1) dy_1 \cdots p(y_n) dy_n.
  \end{split} \]
Here the $x_n$ on the left is a dummy variable; whereas $x_n$ on the right is recursively defined by $x_m=f(x_{m-1})+y_m$, so $x_n$ is a function of the dummy variables $x_{0}, y_{1}, \cdots, y_{n}$.
\end{lemma}

\begin{remark*}
  The integrated formula on the right prescribes how to compute the left integration by Monte-Carlo method.
  That is, for each $l$, we generate random $x_{l,0}$ according to density $h_0$, and $y_{l,1}, \cdots, y_{l,n}$ i.i.d according to $p$.
  Then we compute $\Phi(x_n)$ for this particular sample of $ \left\{x_{l,0}, y_{l,1}, \cdots, y_{l,n} \right\} $.
  Note that the experiments for different $l$ should be either independent or decorrelated.
  Then the Monte-Carlo integration is simply
\[ \begin{split}
  \int \Phi(x_n) ((L_pL_f)^n h_0) (x_n) dx_n
  = \lim_{L\rightarrow\infty} \frac 1L \sum_{l=1}^L \Phi(x_{l,n}).
\end{split} \]
Almost surely according to $h_0\times p_1\times\cdots\times p_n$.
  In this paper we shall use $n,m$ to indicate time steps, whereas $l$ labels different samples.
\end{remark*}

\begin{proof}
Sequentially apply the definition of $L_p$ and $L_f$, we get
\begin{equation} \begin{split} \label{e:year}
    \int \Phi(x_n) ((L_pL_f)^n h_0) (x_n) dx_n \\
    = \iint \Phi(y_n+z_n) (L_f(L_pL_f)^{n-1} h_0) (z_n) p(y_n) dy_n dz_n \\
    = \iint \Phi(x_n(x_{n-1},y_n)) ((L_pL_f)^{n-1} h_0) (x_{n-1}) p(y_n) dy_n dx_{n-1}.
\end{split} \end{equation}
  Here $x_n:=y_n+f(x_{n-1})$ in the last expression is a function of the dummy variables $x_{n-1}$ and $y_n$.
  Roughly speaking, the dummy variable $z_n$ in the second expressions is $f(x_{n-1})$.
  
  Recursively apply \cref{e:year} once, we have 
  \[ \begin{split}
    \int \left(\int \Phi(x_n(x_{n-1},y_n)) ((L_pL_f)^{n-1} h_0) (x_{n-1}) dx_{n-1} \right) p(y_n) dy_n\\
    =
    \int \Phi(x_n(x_{n-1}(x_{n-2},y_{n-1}),y_n)) ((L_pL_f)^{n-2} h_0) (x_{n-2}) p(y_{n-1})p(y_n) dy_{n-1} dy_n dx_{n-2}.
  \end{split} \]
  Here $x_{n-1}:=y_{n-1}+f(x_{n-2})$, so $x_n$ is a function of the dummy variables $x_{n-2},y_{n-1}$, and $y_n$.
  Keep applying \cref{e:year} recursively, we get
\begin{equation} \begin{split}
    \int \Phi(x_n) ((L_pL_f)^n h_0) (x_n) dx_n \\
    = \int \Phi(x_n)  h_0 (x_{0}) p(y_1)\cdots p(y_n) dy_1\cdots dy_n dx_{0},
\end{split} \end{equation}
where $x_n$ is a function as stated in the lemma. 
\end{proof}

Then we give formulas for perturbations.
We first give a pointwise formula for $\delta (L_p L_{f_\gamma} h) (x)$, which seems to be a long-known wisdom.
An intuitive explanation of the following lemma is given in \cref{s:intui}.

\begin{lemma} \label{l:delta}
For any $C^2$ densities $h$ and $p$, any $C^2$ map $f$, 
if $\gamma \mapsto f_\gamma$ is $C^1$ from $\R$ to the space of $C^2$ maps,
then
\[ \begin{split}
  \delta (L_p L_{f_\gamma} h) (x_1)
  = \int -  \delta f_\gamma (x_0)\cdot dp(x_1 - f x_0) \, h(x_0) dx_0.
\end{split} \]
Here $\delta:=\partial/\partial \gamma|_{\gamma=0}$, 
and $\delta f_\gamma (x_0) \cdot dp(x_1 - f x_0)$  is the derivative of the function $p(\cdot)$ at $x_1-fx_0$ in the direction $\delta f_\gamma (x_0)$, which is a vector at $fx_0$.
\end{lemma}

\begin{remark*}[]
Note that we do not compute $\delta L_{f_\gamma} $ separately.
In fact, the main point here is that the convolution with $p$ allows us to differentiate $p$, which is typically much more forgiving than differentiating only $L_{f_\gamma} $.
Although there is a formula for $\delta L_{f_\gamma}$, $\delta L_{f_\gamma} h = \div (h \delta f_\gamma \circ f^{-1})$, but the derivative of $h$ is still unknown.
Moreover, if we quench a specific noise sequence $ \left\{  Y_n=y_n\right\}_{n\ge 0}$, we can define quenched stable and unstable subspaces for this noise sample; then we can give an equivariant divergence formula for the unstable $\delta L_{f_\gamma} h$, where everything is known and can be computed recursively \cite{TrsfOprt}.
However, the equivariant divergence formula requires hyperbolicity, which can be too strict to real-life systems.
\end{remark*}

\begin{proof}
First write a pointwise expression for $L_p L_{f_\gamma} h$: by definition of $L_p$ in \cref{e:Lppt},
\[ \begin{split}
  (L_p L_{f_\gamma} h) (x_1)
  = \int (L_{f_\gamma} h) (z_1) p(x_1-z_1) dz_1.
\end{split} \]
Since $p\in C^2(\R^M)$, we can substitute $p$ into $\Phi$ in the definition of $L_f$ in \cref{e:Lh}, to get
\[ \begin{split}
  = \int  h(x_0) p(x_1- f_\gamma x_0) dx_0.
\end{split} \]
Differentiate with respect to $\gamma$, we have
\[ \begin{split}
  \delta (L_p L_{f_\gamma} h) (x_1)
  = \int - h(x_0) \delta f_\gamma (x_0) \cdot dp(x_1 - f x_0) dx_0.
\end{split} \]
\end{proof} 

Then we give the integral version of \cref{l:delta} which can be computed by Monte Carlo algorithms.
An intuition of the following lemma is given in \cref{s:intui}.

\begin{lemma} \label{l:niu}
Under assumptions of \cref{l:delta}, also fix a bounded $C^2$ observable function $\Phi$, then
\[ \begin{split}
  \int \Phi(x_1) \delta (L_p L_{f_\gamma} h) (x_1) dx_1
  = \iint - \Phi(x_1) \delta f_\gamma (x_0) \cdot \frac {dp(y_1)}{p(y_1)} \, h(x_0) dx_0 \, p(y_1) dy_1.
\end{split} \]
Here $x_1$ in the right expression is a function of the dummy variables $x_0$ and $y_1$, that is, $x_1=y_1+fx_0$.
\end{lemma}

\begin{remark*}
  The point is, if we want to integrate the right expression by Monte-Carlo, just generate random pairs of $ \left\{x_{l,0}, y_{l,1}\right\} $, then compute $x_{l,1}$ and $I_l := \Phi(x_1)\delta f_\gamma x_{l,0} \cdot \frac {dp}{p}( y_{l,1})$, and then average over many $I_l$'s.
\end{remark*} 

\begin{proof}
  Substitute \cref{l:delta} into the integration, we get
\[ \begin{split}
  \int \Phi(x_1) \delta (L_p L_{f_\gamma} h) (x_1) dx_1
  = - \iint \Phi(x_1) \delta f_\gamma (x_0) \cdot dp(x_1 - f x_0) \, h(x_0) dx_0 dx_1
\end{split} \]
The problem with this expression is that, should we want Monte-Carlo, it is not obvious which measure we should generate $x_1$'s according to.
To solve this issue, change the order of the double integration to get
\[ \begin{split}
  = - \int \delta f_\gamma(x_0)\cdot \left( \int \Phi(x_1) dp(x_1 - f x_0)  dx_1 \right) h(x_0) dx_0.
\end{split} \]
Change the dummy variable of the inner integration from $x_1$ to $y_1=x_1-fx_0$, 
\[ \begin{split}
  = - \int \delta f_\gamma(x_0)\cdot \left( \int \Phi(y_1+fx_0) dp(y_1) dy_1 \right) h(x_0) dx_0 \\
  = - \iint \Phi(x_1(x_0,y_1))\, \delta f_\gamma(x_0)\cdot \frac {dp}p (y_1) \, p(y_1) dy_1 h(x_0) dx_0.
\end{split} \]
Here $x_1$ is a function as stated in the lemma.
\end{proof}

Then we combine \cref{l:niu} and \cref{l:PhiMC} to get integrated formulas for perturbations of $ h_T:=  ( L_p L_{f_\gamma})^T h_0$, which can be sampled by Monte-Carlo type algorithms.
Note that here $h_0$ is fixed and does not depend on $\gamma$.

\goldbach*

\begin{remark*}
  We shall explain in detail how to Monte-Carlo integrate the right expression in the next section.
\end{remark*}

\begin{proof}
 Note that $\Phi$ is fixed, use \cref{e:xu}, and let $m=T-n-1$, we get
\[ \begin{split}
  \delta \int \Phi(x_T) h_T(x_T) dx_T
  = \int \Phi(x_T) (\delta h_T) (x_T) dx_T \\
  = \int \Phi(x_T) \sum_{m=0}^{T-1} \left( ( L_p L_f )^{T-m-1} \, \delta(L_p L_{f_\gamma} ) \, (L_p L_f )^{m} h_0 \right) (x_T) dx_T.
\end{split} \]
   
For each $m$, first apply \cref{l:PhiMC} several times,
\[ \begin{split}
  \int \Phi(x_T) \left( ( L_p L_f )^{T-m-1} \, \delta(L_p L_{f_\gamma} ) \, (L_p L_f )^{m} h_0 \right) (x_T) dx_T \\
  = \int \Phi(x_T) \left( \delta(L_p L_{f_\gamma} ) \, (L_p L_f )^{m} h_0 \right) (x_{m+1}) \, dx_{m+1} (pdy)_{m+2\sim T} 
\end{split} \]
Here $x_T$ 
is a function of $x_{m+1}, y_{m+2}, \cdots, y_T$.
Then apply \cref{l:niu} once,
\[ \begin{split}
  = - \int \Phi(x_T) \delta f_\gamma x_m \cdot \frac {dp}p (y_{m+1}) \left( (L_p L_f )^{m} h_0 \right) (x_{m}) \, dx_{m} (pdy)_{m+1\sim T} 
\end{split} \]
Now $x_T$ is a function of $x_{m}, y_{m+1}, \cdots, y_T$.
Then apply \cref{l:PhiMC} several times again, 
\[ \begin{split}
  = - \int \Phi(x_T) \delta f_\gamma x_m \cdot \frac {dp}p (y_{m+1}) h_0 (x_0) \, dx_0 (pdy)_{1\sim T} 
\end{split} \]
Then sum over $m$ to prove the theorem.
\end{proof}

Note that we can subtract any constant from $\Phi$ and does not change the linear response: this is straightforward to understand from a functional point of view, but can also be proved from a more concrete \cref{l:int0}.
Hence, we can centralize $\Phi$, i.e. replacing $\Phi(\cdot)$ by $\Phi(\cdot)-\Phi_{avg,T}$, where the constant
\[ \begin{split}
  \Phi_{avg,T} := \int \Phi(x_T) h_T(x_T) dx_T.
\end{split} \]
We sometimes centralize by subtracting $\Phi_{avg}:= \int \Phi h dx$.
The centralization reduces the amplitude of the integrand, so the Monte-Carlo method converges faster.

\begin{lemma} \label{l:int0}
For any $m\ge0$, if $\lim_{y\rightarrow\infty}p(y)=0$ and  $p$ is $C^1$, then
\[ \begin{split}
  \int  \delta f_\gamma x_{m} \cdot \frac {dp}{p} (y_{m+1})  h_0(x_0) dx_0 (pdy)_{1\sim m+1}
  =
  0.
\end{split} \] 
\end{lemma}

\begin{proof}
  Just notice that $x_m$ is a function of dummy variables $x_0, y_1, \cdots, y_m$, so we can first integrate
\[ \begin{split}
  \delta f_\gamma x_{m} \cdot \int \frac {dp}{p} (y_{m+1})  p(y_{m+1}) dy_{m+1}
  = \delta f_\gamma x_{m} \cdot \int dp (y_{m+1}) dy_{m+1}
  = 0.
\end{split} \]
Since $p(y)\rightarrow 0$ as $y\rightarrow\infty$.
Here $dp$ is the differential of the function $p$, whereas $dy$ indicates the integration.
\end{proof}

For the convenience of computer coding, we explicitly rewrite this theorem into centralized and time-inhomogeneous form, where  $\Phi$, $f$, and $p$ are different for each step.
This is the setting for many important cases, such as finitely-deep neural networks.
The following proposition can be proved by rerunning our previous proofs, while explicitly involving the time-dependency.
Note that we can reuse a lot of data should we also care about the perturbation of the averaged $\Phi_n$ for other layers $1\le n\le T$.

In the following proposition, let $\Phi_T:\R^{M_T}\rightarrow\R$ be the observable function defined on the last layer of dynamics.
Let $h_T$ be the pushforward measure given by the dynamics in \cref{e:dynamicinhomo}, that is, defined recursively by $h_n:=L_{p_n} L_{f_{\gamma,n-1}} h_{n-1}$;
Let $(pdy)_{1\sim T}:=p_1(y_1)dy_1\cdots p_T(y_T)dy_T$ be the independent but not necessarily identical distribution of $Y_n$'s.

\begin{proposition} [No-propagate formula for $\delta h_T$ of time-inhomogeneous systems] \label{l:dhTn}
If $f_{\gamma,n}$ and $p_n$ satisfy assumptions of \cref{l:niu,l:int0} for all $n$, then
\[ \begin{split}
  &\delta \int \Phi_T(x_T) h_T(x_T) dx_T
  \\
  = &- \int \left(\Phi_T(x_T) - \Phi_{avg,T} \right) \sum_{m=0} ^ {T-1} \left( \delta f_{\gamma,m} x_{m} \cdot \frac {dp}{p} (y_{m+1}) \right) h_0(x_0) dx_0 (pdy)_{1\sim T}.
\end{split} \]  
Here $\Phi_{avg,T} := \int \Phi_T h_T $.
\end{proposition}

For the perturbation of physical measures (assuming the limit $h\weakliml h_T$ exists), the Monte-Carlo formula for $\delta h$ can further take the form of long-time average on an orbit.
Because $h$ is invariant under the dynamics, so we can apply the ergodic theorem, forget details of $h_0$, and sample $h$ by points from a long orbit: this is the Monte-Carlo method when the background measure $h$ is given by dynamics.
Of course, here we require the dynamics being time-homogeneous ($f$ and $p$ are the same for each step) to invoke ergodic theorems.

\begin{proposition}[No-propagate orbitwise formula for $\delta h$] 
\label{l:dh}
With same assumptions of \cref{l:niu,l:int0}, 
also assume that the physical measure's density $h$  exists for the dynamic $X_{n+1}=f(X_n)+Y_{n+1}$, and $\delta h = \sum_{n\ge0} (L_f L_p)^{n} \delta(L_p L_{f_\gamma} ) h$, then
\[ \begin{split}
  \delta \int \Phi(x) h(x) dx
  \,\overset{\mathrm{a.s.}}{=}\, 
  \lim_{W\rightarrow\infty} \lim_{L\rightarrow\infty} - \frac 1L \sum_{n=0}^W  \sum_{l=1}^L \left( \Phi(X_{n+1+l}) - \Phi_{avg} \right) \, \delta f_\gamma X_l \cdot \frac{dp}p(Y_{1+l})
\end{split} \]  
almost surely according to the measure obtained by starting the dynamic with $X_0, Y_1, Y_2,\cdots \sim h\times p\times p \times \cdots$.
Here $\Phi_{avg}:=\int\Phi h$.
\end{proposition}

\begin{remark*}
(1)
So far, for all cases where existence of $h$ and $\delta h$ can be proved, the formulas are all equivalent to the one stated in the assumptions (this formula might be the \textit{only} choice).
This paper does not attempt to prove the existence of $h $ and  $ \delta h $.
(2)
Since $p$ is smooth density, the support of $h$ must have positive Lebesgue measure, so we can actually start from, and converge according to, any Lebesgue-equivalent measures on the support of $h$.
\end{remark*} 

\begin{proof}
  By \cref{e:xu}, 
  \[ \begin{split}
    \delta \int \Phi(x) h(x) dx
    = \int \Phi(x) \delta h (x) dx
    \\
    = \lim_{W\rightarrow\infty} \sum_{n=0}^W \int \Phi(x) \left( (L_f L_p)^{n} \delta(L_p L_{f_\gamma} ) h \right) (x) dx
  \end{split} \]
  By the same argument as \cref{t:finiteT},
  \[ \begin{split}
    \int \Phi(x_{n+1}) \left( (L_f L_p)^{n} \delta(L_p L_{f_\gamma} ) h \right) (x_{n+1}) dx_{n+1} \\
    = - \int \Phi(x_{n+1}) \delta f_\gamma x_0 \cdot \frac{dp}p(y_1) h(x_0) dx_0 (pdy)_{1\sim n+1}
  \end{split} \]
  Since $h$ is invariant for the dynamic $X_{n+1}=f(X_n)+Y_{n+1}$, if $X_0\sim h$, then $\{X_l, Y_{l+1},\cdots,Y_{l+n}\}_{l\ge0}$ is a stationary sequence, so we can apply Birkhoff's ergodic theorem (the version for stationary sequences), 
  \[ \begin{split}
    \,\overset{\mathrm{a.s.}}{=}\, 
    \lim_{L\rightarrow\infty} - \frac 1L \sum_{l=1}^L \Phi(X_{n+1+l}) \, \delta f_\gamma X_l \cdot \frac{dp}p(Y_{1+l})
  \end{split} \]
  By substitution we have
  \[ \begin{split}
    \delta \int \Phi(x) h(x) dx
    \,\overset{\mathrm{a.s.}}{=}\, 
    \lim_{W\rightarrow\infty} \lim_{L\rightarrow\infty} - \sum_{n=0}^W \frac 1L \sum_{l=1}^L \Phi(X_{n+1+l}) \, \delta f_\gamma X_l\cdot \frac{dp}p(Y_{1+l})
  \end{split} \]
  Then we can rerun the proof after centralizing $\Phi$.
\end{proof}

\subsection{An intuitive explanation}
\hfill\vspace{0.1in}
\label{s:intui}

We give an intuitive explanation to \cref{l:delta} and \cref{l:niu}.
First, we shall adopt a more intuitive but restrictive notation.
Define $\tf(x,\gamma)$ such that 
\[ \begin{split}
  \tf (f(x), \gamma) = f_\gamma(x)
\end{split} \]
In other words, we write $f_\gamma $ as appending a small perturbative map $\tf$ to $f:=f_{\gamma=0}$.
Here $\tf$ is the identity map when $\gamma=0$.
Note that $\tf$ can be defined only if $f$ satisfies 
\[ \begin{split}
  f_\gamma(x) = f_\gamma(x')
  \quad \textnormal{whenever} \quad 
  f(x) = f(x').
\end{split} \]
For example, when $f$ is bijective then we can well-define $\tf$.
Hence, this new notation is more restrictive than what we used in other parts of the paper.
But it lets us see more clearly what happens during the perturbation.

With the new notation, the dynamics can be written now as
\[ \begin{split}
  X_n =  \tf f(X_{n-1}) + Y_n ,
  \quad \textnormal{where} \quad 
  Y_n \iid  p.
\end{split} \]
By this notation, only $f$ changes time step, but $\tf$ and adding $Y$ do not change time.
In this subsection we shall start from after having applied the map $f$, and only look at the effect of applying $\tf$ and adding noise:
this is enough to account for the essentials of \cref{l:delta} and \cref{l:niu}.
Roughly speaking, the $h'$ we use in the following is in fact $h':=L_f h$;
$h'$ is the density of $z_0:=fx_{-1}$, where $x_{-1}$ is from the previous step, so in the current step $z_0+y_0=x_0$, and we omit the subscript $0$.

With the new notation, \cref{l:delta} is essentially equivalent to
\begin{equation} \begin{split} \label{e:dao}
  \delta (L_p L_{\tf} h') (x)
  = \int -  \delta \tf z \cdot dp(x - z) \, h'(z) dz.
\end{split} \end{equation}
We explain this intuitively by \cref{f:intui}.
Let $z$ be distributed according to $h'$, then $L_p h'$ is obtained by first attach a density $p$ to each $z$, then integrate over all $z$.
$L_p L_\tf h'$ is obtained by first move $z$ to $\tf z$ and then perform the same procedure.
Hence, $\delta L_p L_\tf h'(x)$ is first computing  $\delta p_{\tf z}(x)$ for each $z$, then integrate over $z$.
Let
\[ \begin{split}
   p_{\tf z}(x) := p(x - \tf z)
\end{split} \]
be the density of the noise centered at $\tf z$.
So
\[ \begin{split}
  \delta p_{\tf z}(x)
  = -dp_{\tf z}(x) \cdot \delta \tf z (x)
  = -dp(x - \tf z) \cdot \delta \tf z.
\end{split} \]
Here $\delta\tf z (x)$ in the middle expression is the horizontal shift of  $p_{\tf z}$'s level set previously located at $x$.
Since the entire Gaussian distribution is parallelly moved on the Euclidean space $\R^M$, so $\delta\tf z (x)=\delta\tf z$ is constant for all $x$.
Then we can integrate over $z$ to get \cref{e:dao}.

\begin{figure}[ht] \centering
  \includegraphics[scale=0.4]{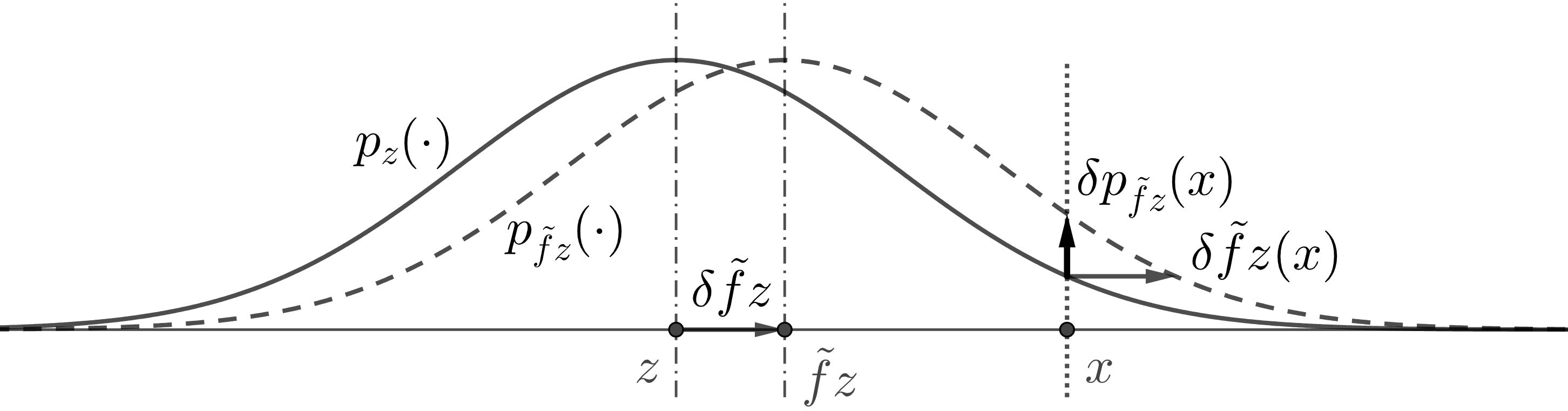}
  \caption{Intuitions for \cref{l:delta} and \cref{l:niu}.}
  \label{f:intui}
\end{figure}

\Cref{l:niu} is equivalent to
\[ \begin{split}
   \delta \int \Phi(x)  (L_p L_\tf h') (x) dx
  = \iint - \Phi(x) \delta \tf (z)\cdot \frac {dp(y)}{p(y)} \, h'(z) dz \, p(y) dy,
\end{split} \]
where $x$ on the right is a function $x=z+y$.
Intuitively, this says that the left side equals to first compute
$\delta \int \Phi(z+y) p_{\tf z}(y) dy $
for each $z$, then integrate over $z$.
The integration on the left uses $x$ as dummy variable, which is convenient for the transfer operator formula above.
But it does not involve a density for $x$, so is not immediately ready for Monte-Carlo.
The right integration is over $h'(z)$ and $p(y)$, which are easy to sample.

It is important that we differentiate only $p$ but not $f$.
In fact, our core intuitions are completely within one time step, and do not even involve $f$.
Hence, we can easily generalize to cases where $f$ is bad, for example, when $f$ is not bijective, or when $f$ is not hyperbolic: these are all very difficult cases for deterministic systems.
Moreover, our algorithm is different from prevailing algorithms, which average some linear response formulas for deterministic systems over many samples.
Since the formula for quenched linear response all involve the derivative of $f$, they are hindered by either the gradient explosion, the dimensionality, or the lack of hyperbolicity.

\subsection{Further generalizations}
\hfill\vspace{0.1in}
\label{s:generalize}

\subsubsection{\texorpdfstring{$p$}{p} depends on \texorpdfstring{$z$}{z} and \texorpdfstring{$\gamma$}{gamma}}
\hfill\vspace{0.1in}


Our pictorial intuition does not care whether $p$ depends on $\gamma$ or $z$. So we can generalize all of our results to the case
\[ \begin{split}
  X_{n+1} = f(X_n) +Y_{n+1},
  \quad \textnormal{where} \quad 
  \Pro(Y_{n+1}=y \, | \, f(X_n)=z) = p_{\gamma,z}(y).
\end{split} \]
We still assume the same regularity for $p$, $\Phi$, and $f$.
The long-time case also further assumes the existence of $h$ and $\delta h$.
We do not repeat the proof; rather, we directly write down these generalizations for future references.

\Cref{e:Lppt} and \cref{e:Lpint} become
\[ \begin{split}
  L_{p_\gamma} h' (x) 
  = \int h'(x-y) p_{\gamma,x-y}(y) dy
  = \int h'(z) p_{\gamma,z}(x-z) dz.
  \\
  \int \Phi (x) L_{p_\gamma} h' (x) dx
  = \iint \Phi(y+z) h'(z) p_{\gamma,z}(y) dz dy.
\end{split} \]
\Cref{l:PhiMC} becomes
  \[ \begin{split}
    \int \Phi(x_n) ((L_{p_\gamma} L_f)^n h_0) (x_n) dx_n
  = \int \Phi(x_n) h_0(x_0) dx_0 p_{\gamma,z_1}(y_1) dy_1 \cdots p_{\gamma,z_n}(y_n) dy_n.
  \end{split} \]
Here $z_m$ and $x_m$ on the right are recursively defined by $z_m=f(x_{m-1})$, $x_m=z_m + y_m$.
We also have the pointwise formula
\[ \begin{split}
  (L_{p_\gamma} L_{f_\gamma} h) (x_1)
  = \int h(x_0) p_{\gamma,f_\gamma x_0}(x_1-f_\gamma x_0) dx_0.
\end{split} \]

Since now $p$ depends on $\gamma$ via three ways, \cref{l:delta} becomes
\[ \begin{split}
  \delta (L_{p_\gamma} L_{f_\gamma} h) (x_1)
  = \int \dd p\gamma (x_0,y_1) h(x_0) dx_0,
  \quad \textnormal{where} \quad 
  z_1 = fx_0,
  \quad
  y_1 = x_1-z_1,
  \\
  \dd p\gamma (x_0,y_1) 
  := \dd {}\gamma p_{\gamma, f_\gamma x_0} (x_1-f_\gamma x_0)
  = \delta p (z_1, y_1) + \delta f_\gamma (x_0)\cdot \left(\pp p z -  \pp p y\right) (z_1,y_1) .
\end{split} \]
Here derivatives $\pp p z$ and $\pp p y$ refer to writing the density as $p_{\gamma,z}(y)$.
If $p$ does not depend on $\gamma$ and $z$, then we recover \cref{l:delta}.
\Cref{l:niu} becomes
\[ \begin{split}
  \delta \int \Phi(x_1) (L_p L_{f_\gamma} h) (x_1) dx_1
  = \int \Phi(x_1) \delta (L_p L_{f_\gamma} h) (x_1) dx_1
  \\
  = \iint \Phi(x_1) 
  \frac {dp}{p d\gamma} (x_0,y_1) h(x_0) dx_0 \, p(y_1) dy_1.
\end{split} \]
Here 
\[ \begin{split}
  \frac {dp}{p d\gamma} (x_m,y_{m+1})
  := \frac 1 { p_{f x_m}(y_{m+1}) }
  \dd {}\gamma p_{\gamma, f_\gamma x_m} (x_{m+1}-f_\gamma x_m)
\end{split} \]
\Cref{t:finiteT} becomes
\[ \begin{split}
  \delta \int \Phi(x_T) h_T(x_T) dx_T
  = \int \Phi(x_T) \sum_{m=0} ^ {T-1} 
  \frac {dp}{p d\gamma} (x_m,y_{m+1}) 
  h_0(x_0) dx_0 (p_z dy)_{1\sim T}.
\end{split} \]  
Here $(p_zdy)_{1\sim T}:=p_{z_1}(y_1)dy_1\cdots p_{z_T}(y_T)dy_T$ is the distribution of $Y_n$'s;
$z_{m+1}=fx_m$ is a function of dummy variables $x_0, y_1, \cdots, y_m$.

We still have free centralization for $\Phi$.
So \cref{l:dhTn} and \cref{l:dh} still hold accordingly,
in particular, when the system is time-homogeneous and has physical measure $h$, we still have
\[ \begin{split}
  \delta \int \Phi(x) h(x) dx
  \,\overset{\mathrm{a.s.}}{=}\, 
  \lim_{W\rightarrow\infty} \lim_{L\rightarrow\infty} - \frac 1L \sum_{n=0}^W  
  \sum_{l=1}^L \left( \Phi(X_{n+1+l}) - \Phi_{avg} \right) \,
  \frac{dp}{pd\gamma}(X_l,Y_{1+l}).
\end{split} \]

\subsubsection{General random dynamical systems}
\hfill\vspace{0.1in}
\label{s:rand}

For general random dynamical systems, at each step $n$, we randomly select a map $g$ from a family of maps, denote this random map by $G$, the dynamics is now
\[ \begin{split}
  X_{n+1} = G(X_n).
\end{split} \]
The selection of $G$'s are independent among different $n$.
It is not hard to see our pictorial intuition still applies, so our work can generalize to this case.

We can also formally transform this general case to the case in the previous subsubsection.
To do so, define the deterministic map $f$
\[ \begin{split}
  f(x) := \E [G(x)],
\end{split} \]
where the expectation is with respect to the randomness of $G$.
Then the dynamic equals in distribution to 
\[ \begin{split}
  X_{n+1} = f(X_n) + Y_{n+1},
  \quad \textnormal{where} \quad 
  Y_{n+1} = F(X_n) - f(X_n).
\end{split} \]
Hence the distribution of $Y_{n+1}$ is completely determined by $X_n$.

The caveat is that we still need to compute the derivatives of the distribution of $Y$, which is equivalent to that of $F$.
This can be an extra difficulty; but sometimes we have to take on this difficulty.
On the other hand, if we start from deterministic case, and use randomness to obtain approximations, then it should be easier to use the additive form and a simple distribution for $Y$.

\section{No-propagate response algorithm}
\label{s:algorithm}

\subsection{Finite \texorpdfstring{$T$}{T} and time-inhomogeneous case}
\hfill\vspace{0.1in}
\label{s:dhT}

We give the procedure list of the no-propagate algorithm for time-inhomogeneous $\delta h_T$ in \cref{l:dhTn}.
Here $L$ is the number of sample paths whose initial conditions are generated independently from $h_0$.
This algorithm requires that we already have a random number generator for $h_0$ and  $p$:
this is typically easier for $p$ since we tend to use simple $p$ such as Gaussian; but we might ask for specific $h_0$, which requires more advanced sampling tools.

\vspace{0.2in}\hrule\vspace{0.05in}
\begin{algorithmic}
\Require $L$, random number generators for densities $h_0$ and $p_m$.
\For{$l = 1, \dots, L$}
  \State Independently generate sample $x_0$ from $X_0\sim h_0$, 
  \For{$m = 0, \dots, T-1$}
    \State Independently generate sample $y_{m+1}$ from $Y_{m+1}\sim p_{m+1}$
    \State $I_{l,m+1} \gets \delta f x_{l,m} \cdot \frac {dp}{p} (y_{m+1}) $
    \State $x_{l,m+1} \gets f(x_{l,m}) + y_{m+1}$
  \EndFor
  \State $\Phi_{l,T}\gets\Phi_T (x_{l,T})$
  \Comment $\Phi_{l,T}$ takes less storage than $x_{l,T}$.
  \State $S_{l} \gets - \sum_{m=1}^{T}  I_{l,m}$
  \Comment No need to store $x_m$ or $y_m$.
\EndFor
\State{$\Phi_{avg,T} := \int \Phi_T h_{\gamma,T} dx 
\approx \frac 1L \sum_{l=1}^L \Phi_{l,T}$
}
\Comment{
Evaluated at $\gamma=0$.
}
\State $\delta \Phi_{avg,T} \approx \frac 1L \sum_{l=1}^L S_l \left( \Phi_T (x_{l,T}) -\Phi_{avg,T} \right)$ 
\end{algorithmic}
\vspace{0.05in}\hrule\vspace{0.2in}

This is automatically an adjoint algorithm.
In fact, the notion of `adjoint algorithms' does not quite apply to no-propagate algorithms, since we do not compute the Jacobian matrix at all, so the most expensive operation per step is computing $f(x_m)$, which is not much more expensive than computing $\delta f(x_m)$.
The cost for a new parameter (i.e. a new $\delta f$) in this algorithm equals  that cost in any other adjoint algorithms, which is cheaper than the first parameter of any other algorithms which involve forward or backward propagations.

We remind readers of some useful formulas when we use Gaussian noise in $\R^M$.
Let the mean be $\mu\in\R^M$, the covariance matrix be $\Sigma$, which is defined by
$ \Sigma_{i,j} := \operatorname{E} [(Y_i - \mu_i)( Y_j - \mu_j)] = \operatorname{Cov}[Y_i, Y_j] $.
Then the density is
\[ \begin{split}
p (y) = (2\pi)^{-k/2}\det (\Sigma)^{-1/2} \, \exp \left( -\frac{1}{2} (y - \mu)^\mathsf{T} \Sigma^{-1}(y - \mu) \right)
\end{split} \]
Hence we have
\[ \begin{split}
\frac{dp}p (y)
= - \Sigma^{-1} (y-\mu) 
\in \R^M.
\end{split} \]
Typically we use normal Gaussian, so $\mu = 0$ and  $\Sigma = \sigma^2 I_{M\times M}$, so we have
\[ \begin{split}
\frac{dp}p (y)
= - \frac 1{\sigma^2} y
\in \R^M.
\end{split} \]

\subsection{Infinite time and time-homogeneous case}
\hfill\vspace{0.1in}

We give the procedure list of the no-propagate algorithm for time-homogeneous $\delta h$ in \cref{l:dh}.
Here $M_{pre}$ is the number of preparation steps, during which the background measure is evolved, so that $x_0$ is distributed according to the physical measure $h$.
Here $W$ is the decorrelation step number, typically $W \ll L$, where $L$ is the orbit length.
Since here $h$ is given by the dynamic, we only need to sample the easier density $p$.

\vspace{0.2in}\hrule\vspace{0.05in}
\begin{algorithmic}
\Require $M_{pre}, W\ll L$, random number generator for $p$.
\State Take any $x_0$.
\For{$m = 0, \cdots, M_{pre}$} 
  \Comment{To land $x_0$ onto the physical measure.}
  \State Independently generate sample $y$ from density $h$
  \State $x_0 \gets f(x_0)+y$
\EndFor
\For{$m = 0, \cdots, W+L-1$}
    \State Independently generate sample $y_{m+1} $ from $Y_{m+1} \sim p$
    \State $I_{m+1} \gets \delta \tf (f x_{m})\cdot \frac {dp}{p} (y_{m+1}) $
    \State $x_{m+1} \gets f(x_m)+y_{m+1} $
    \State $\Phi_{m+1} \gets \Phi(x_{m+1})$
\EndFor
\State $\Phi_{avg}:=\int \Phi h_{\gamma} dx\approx \frac 1L \sum_{l=1}^L \Phi_l$
\State $\delta \Phi_{avg} \approx  - \frac 1L \sum_{n=0}^W  \sum_{l=1}^L \left(\Phi_{n+l} - \Phi_{avg} \right) I_{l}$ 
\end{algorithmic}
\vspace{0.05in}\hrule\vspace{0.2in}

\section{Numerical examples}
\label{s:examples}

This section demonstrates the no-propagate response algorithm on several examples.
We no longer label the subscript $\gamma$, and $\delta$ can be the derivative for nonzero $\gamma$; the dependence on $\gamma$ should be clear from context.
The codes used in this section are all at \url{https://github.com/niangxiu/np}.

\subsection{Tent map with elevating center}
\hfill\vspace{0.1in}
\label{s:failEx}

We demonstrate the algorithm on the tent map.
The dynamics is
\[ \begin{split}
  X_{n+1}=f_\gamma(X_n)+Y_{n+1}
  \quad \textnormal{mod 1,} \quad 
  \quad \textnormal{where} \quad 
  \\
  Y_n\iid \cN(0,\sigma^2),
  \quad \textnormal{} \quad 
  f_\gamma(x) = \begin{cases}
    \gamma x 
    \quad\text{if}\quad 0\le x\le 0.5,
    \\
    \gamma(1-x)
    \quad\text{otherwise.}
  \end{cases}
\end{split} \]
In this subsection, we fix the observable
\[ \begin{split}
\Phi(x)=x
\end{split} \]

Previous linear response algorithms based on randomizing algorithms for deterministic systems do not work on this example.
In fact, it is proved by Baladi that for the deterministic case ($\sigma=0$), $L^1$ linear response does not exist \cite{baladi07}.
The fast response algorithm for deterministic systems fails on this example, due that Monte-Carlo integration does not work on the second derivative $f''$, which is a delta functions.
The ensemble method also fails due to the exploding gradient.

We shall demonstrate the no-propagate algorithm on this example, and show that the noise case can give a useful approximate linear response to the deterministic case.
In the following discussions, unless otherwise noted, the default values for the (hyper-)parameters are
\[ \begin{split}
 \gamma=3, \quad
 \sigma=0.1, \quad 
 W=7. 
\end{split} \]

\begin{figure}[ht] \centering
  \includegraphics[height=7cm]{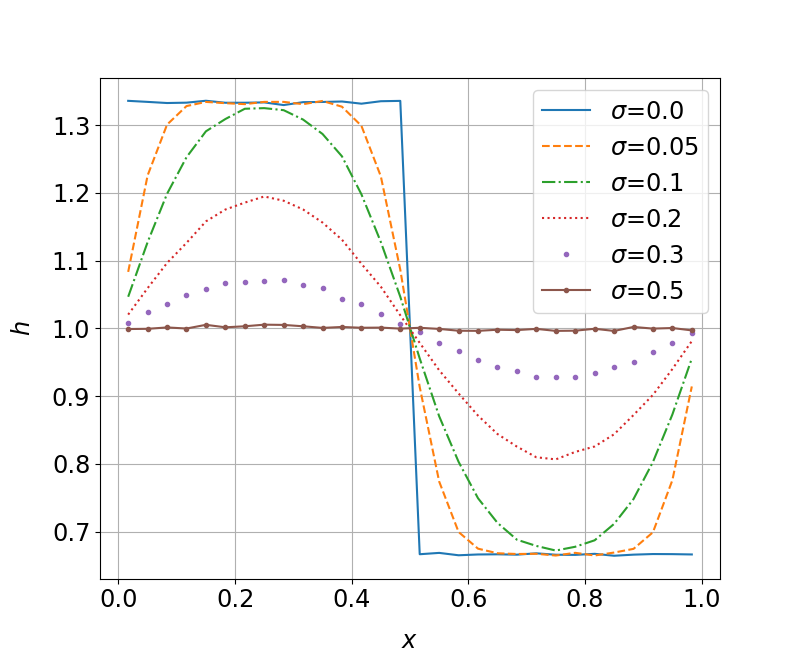} 
  \caption{
  Densities $h$ of different $\sigma$. 
  Here $L=10^7$.
  }
  \label{f:tentden}
\end{figure}

We first test the effect of adding noise on the physical measure.
As shown in \cref{f:tentden}, the density converges to the deterministic case as $\sigma\rightarrow 0$.
This and later numerical results indicate that we can wish to find some approximate `linear response' for the deterministic system, which provides some useful information about the trends between the averaged observable and $\gamma$.

\begin{figure}[ht] \centering
  \includegraphics[scale=0.4]{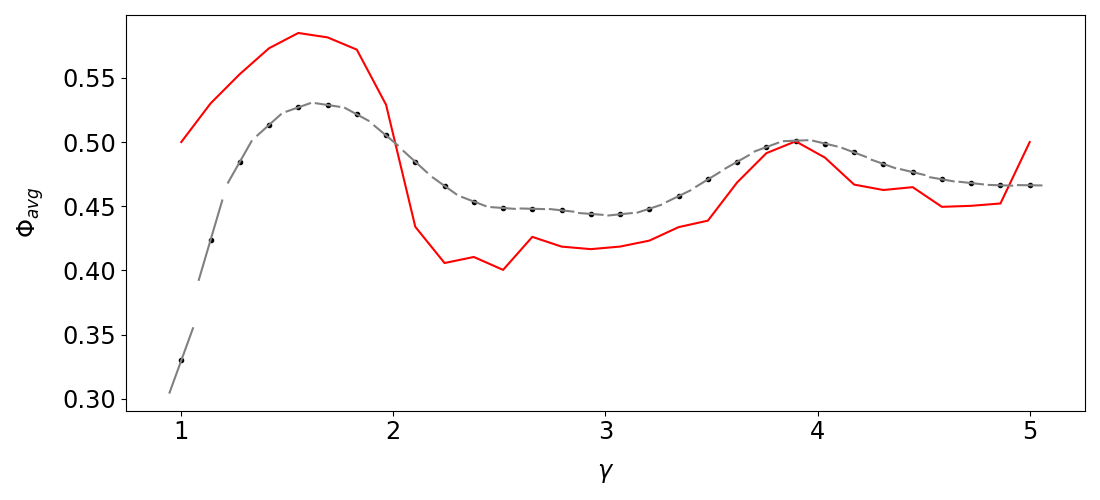}
  \caption{$\Phi_{avg}$ and $\delta \Phi_{avg} $ for different parameter $\gamma$.
  Here $L=10^6$ for all orbits.
  The dots are $\Phi_{avg}$, and the short lines are $\delta \Phi_{avg} $ computed by the no-propagate response algorithm; they are computed from the same orbit.
  The long red line is $\Phi_{avg}$ when there is no noise ($\sigma=0$).
  }
\end{figure}

Then we run the no-propagate algorithm  to compute linear responses for different $\gamma$.
This is the main scenario where the algorithm is useful, that is, we want to know the relation between $\Phi_{avg}=\int \Phi h dx$ and $\gamma$.
As we can see, the algorithm gives an accurate linear response for the noisy case.
Moreover, the linear response in the noisy case is an reasonable reflection of the observable-parameter relation of the deterministic case.
Most importantly, the local min/max of the noisy and deterministic cases tend to locate similarly.
Hence, we can use the no-propagate algorithm of the noisy case to help optimizations of the deterministic case.

\begin{figure}[ht] \centering
  \includegraphics[width=0.45\textwidth]{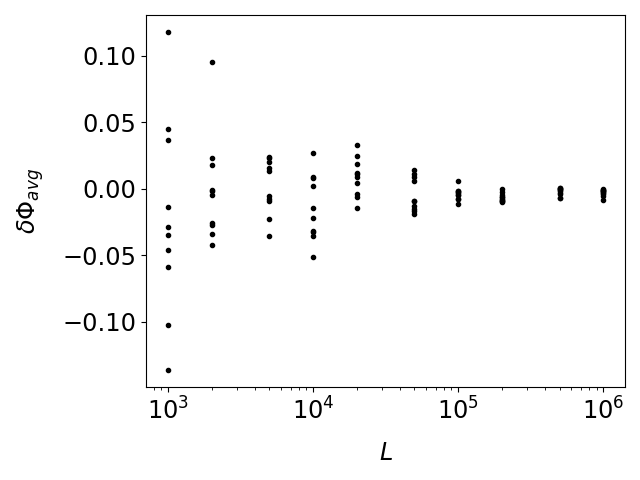}
  \includegraphics[width=0.45\textwidth]{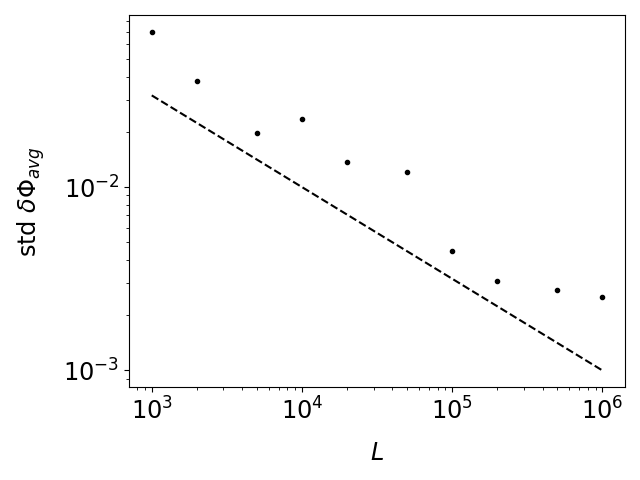}
  \caption{Effects of orbit length $L$.
  Left: derivatives from 10 independent computations for each $L$.
  Right: the standard deviation of the computed derivatives, where the dashed line is $L^{-0.5}$.}
  \label{f:tent_L}
\end{figure}

Then we show the convergence of the no-propagate algorithm with respect to $L$ in \cref{f:tent_L}.
In particular, the standard deviation of the computed derivative is proportional to $L^{-0.5}$.
This is the same as the classical Monte Carlo method for integrating probabilities, with $L$ being the number of samples.
This is expected: since here we are given the dynamics, so our samples are canonically from a long orbit.

\begin{figure}[ht] \centering
  \includegraphics[width=0.45\textwidth]{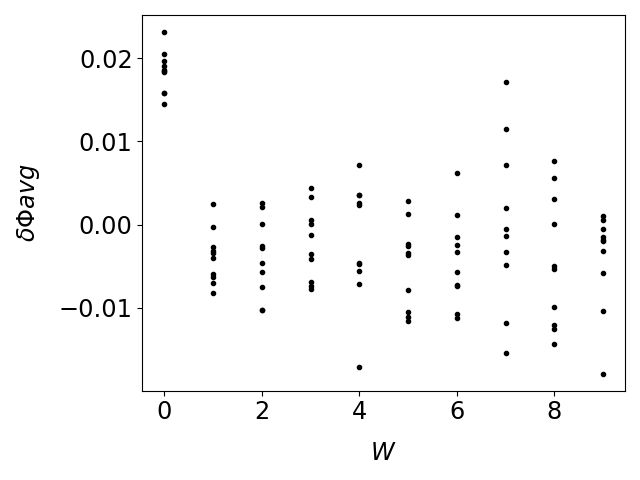}
  \includegraphics[width=0.45\textwidth]{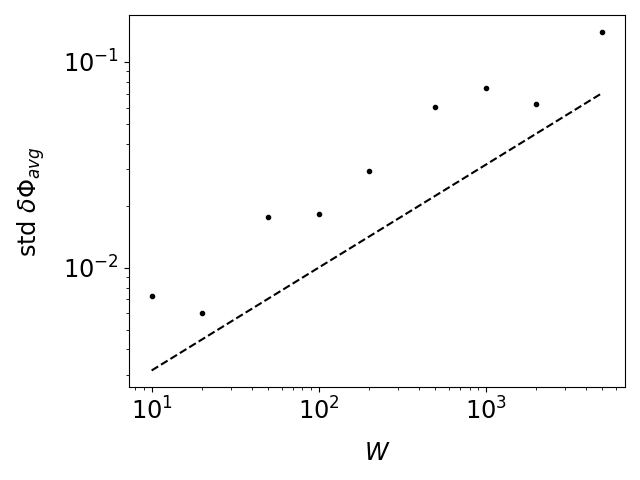}
  \caption{Effects of $W$. 
  Here $L=10^5$.
  Left: derivatives computed by different $W$'s.
  Right: standard deviation of derivatives, where the dashed line is $0.001W^{0.5}$.}
  \label{f:tent_W}
\end{figure}

\Cref{f:tent_W} shows that the bias in the average derivative decreases as $W$ increases, but the standard deviation increases roughly like $O(W^{0.5})$.
Note that if we do not centralize $\Phi$, then the standard deviation would be like $O(W)$.

\subsection{A chaotic neural network}
\hfill\vspace{0.1in}
\label{s:21dim}

Many work in machine learning aims to suppress gradient explosion, which is essentially the definition of chaos.
However, even with modern architecture, there is no guarantee that gradient explosion (or chaos) can be precluded.
In this subsection we use the no-propagate algorithm to compute the linear response of a chaotic neural network.
Unlike stochastic back-propagation method, we do not perform propagation at all, so we are not hindered by gradient explosion.
But we need to assume noise at each layer; this may introduce a systematic error, if the original model does not have such noise.

Our phase spaces are $X_n,Y_n\in\R^9$, $n=0,1,\cdots, T$, where $T=50$, and the dynamic is 
\[ \begin{split}
  X_{n+1}=f_\gamma(X_n)+Y_{n+1}
  \quad \textnormal{where} \quad 
  Y_n\iid \cN(0,0.5^2 I),
  \\
  f_\gamma(x) = J \tanh (x + \gamma \one),
  \quad \textnormal{} \quad 
  J = CJ_0,
  \quad \textnormal{where} \quad 
  \\
  J_0 = 
  \left[\begin{array}{ccccccccc}
 -0.54 &-1.19 &-0.33 & 1.66 &-0.5  &-1.3  & 1.52 &-0.5  & 1.95 \\
 -1.6  &-1.55 &-1.45 & 0.61 & 1.92 & 0.59 &-0.16 &-1.14 &-1.27 \\
 -0.59 &-0.65 &-1.32 &-1.46 &-0.82 &-0.95 &-1.47 &-0.08 &-0.38 \\
 -0.78 &-0.26 & 0.87 & 1.99 & 0.07 & 0.87 &-0.79 &-0.44 & 1.11 \\
  0.8  &-1.28 &-0.52 &-1.01 & 1.49 & 1.49 &-1.65 &-0.45 & 0.21 \\
 -1.77 & 0.03 &-1.39 &-0.28 & 0.44 & 1.27 & 0.61 & 0.01 &-0.02 \\
 -0.18 &-0.29 &-0.73 & 0.53 &-0.82 &-1.58 &-1.41 & 0.07 &-1.84 \\
  0.64 & 0.86 & 0.73 & 0.96 &-0.06 & 0.04 & 1.1  & 1.22 &-0.28 \\
  1.18 &-1.95 &-0.37 & 0.01 & 1.24 &-0.32 & 0.43 & 0.06 &-1.28
  \end{array}\right].
\end{split} \]
Here $I$ is the identity matrix, $\one=[1,\cdots, 1]$, and for $x =[x_1,\cdots,x_8], x_i\in\R$,
\[ \begin{split}
  \tanh(x)
  := [\tanh(x_1),\cdots,\tanh(x_8)].
\end{split} \]
We set $X_0 \sim \cN(0,I)$ and $h_0$ be its density.

There is a somewhat tight region for $C$ such that the system is chaotic: when $C<1$, then $J$ is small so the Jacobian is small; when $C>10$, then the points tend to be far from zero, so the derivative of $\tanh$ is small, so the Jacobian is also small.
Our choice $C=4$ gives roughly the most unstable network, and we roughly compare with the cost of the ensemble, or stochastic gradient, formula of the linear response, which is
\[ \begin{split}
  \delta \E (\Phi)
  = \E \left(\sum_{n=1}^{50} \delta f(x_{50-n}) \cdot f^{*n} d\Phi(x_{50})\right)
\end{split} \]
Here $f^*$ is the backpropagation operator for covectors, which is the transpose of the is the Jacobian matrix $Df$.
On average $|f^{*50}| \approx 10^4\sim10^5$, $d\Phi\cdot \delta f\approx 10$, so the integrand's size is about $10^5\sim 10^6$.
This would require about $10^{10} \sim 10^{12}$ samples (a sample is a realization of the 50-layer network) to reduce the sampling error to $O(1)$.

\begin{figure}[ht] \centering
  \includegraphics[scale=0.4]{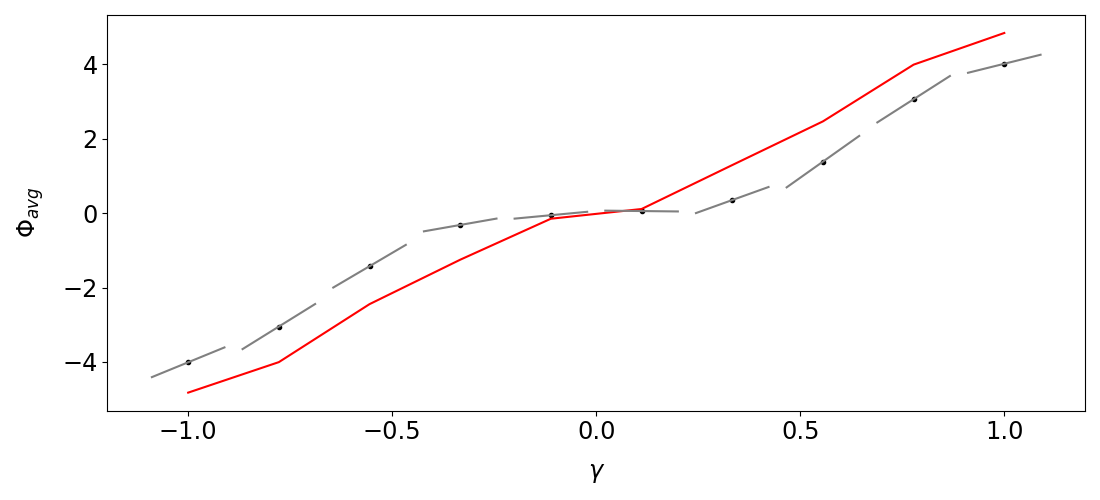}
  \caption{$\Phi_{avg}$ and $\delta \Phi_{avg} $ for different parameter $\gamma$.
  Here the total number of samples is $L=10^5$.
  The long red line is $\Phi_{avg}$ when there is no noise ($\sigma=0$).
  }
  \label{f:NN}
\end{figure}

\Cref{f:NN} shows the result of the no-propagate algorithm on this example.
The algorithm accurately reveals the derivative of our problem, which is also a good reflection of the parameter-objective trend of the zero-noise case.
The total time cost for running the algorithm on $L=10^5$ orbits to get a linear response, using a 1 GHz computer thread, is 60 seconds.

Our model is modified from its original form in \cite{Cessac04,Cessac06}.
In the original model, the entries of the weight matrix $J$ are randomly generated according to certain laws; as discussed in \cref{s:rand}, we can rewrite this randomness as an additive noise field, then obtain exact solutions to the original problem.
We can also further generalize this example to time-inhomogeneous cases.

Finally, we acknowledge that our neural network's architecture is outdated, but modern architectures are not good tests for our algorithm.
Because the backpropagation method does not work in chaos, current architectures typically avoid chaos.
With our work, we might potentially have much more freedom in choosing architectures beyond the current ones.

\section{Discussions}
\label{s:discuss}

The no-propagate algorithm is robust, does not have systematic error, has very low cost per step, and is not cursed by dimensionality.
But there is a caveat: when the noise is small, we need many data for the Monte-Carlo method to converge.
Hence we can not expect to use the small $\sigma$ limit to get an easy approximation of the linear response for deterministic systems.
In this section, we first give a rough cost-error estimation of the problem, then discuss how to potentially reduce the cost by further combining with the fast response algorithm, which was developed for deterministic linear responses.
We also compare with some seemingly similar algorithms for linear responses.

\subsection{A very rough cost-error estimation}
\hfill\vspace{0.1in}
\label{s:costErr}

When the problem is intrinsically random, the scale of noise $\sigma$ has been fixed.
For this case, there are two sources of error.
The first is due to using a finite decorrelation step number $W$; this error is $O(\theta^W)$ for some $0<\theta<1$.
The second is the sampling error due to using a finite number $L$ of samples.
Assume we use Gaussian noise in \cref{t:finiteT}, since we are averaging a large integrand to get a small number, we can approximate the standard deviation of the integrand by its absolute value.
The integrand is the sum of $W$ copies of $\frac{dp}p = -\frac y{\sigma^2}\sim \frac 1\sigma$, so the size is roughly $\sqrt W /\sigma$.
So the sampling error is $\sqrt W /\sigma\sqrt L$, where $L$ is the number of samples.
Together we have the total error $\eps$
\[ \begin{split}
  \eps = O(\theta^W) + 
  O \left(\frac {\sqrt W} {\sigma\sqrt L} \right).
\end{split} \]

This gives us a relation among $\eps, W$, and $L$, where $L$ is proportional to the overall cost.
In practice, we typically set the two errors be roughly equal,
which gives an extra relation for us to eliminate $W$ and obtain the cost-error relation
\[ \begin{split}
  \theta^W = \frac {\sqrt W} {\sigma\sqrt L},
  \quad \textnormal{so} \quad 
  L 
  = O \left( \frac { W} {\sigma^2 \theta^{2W}} \right)
  = O \left( \frac {\log _\theta \eps} {\sigma^2 \eps^{2}} \right)
  .
\end{split} \]
This is rather typical for Monte-Carlo method.
But the problem is that cost can be large for small $\sigma$.

On the other hand, if we use random system to approximate deterministic systems, then we can have the choice on the noise scale $\sigma$.
Now each step further incurs an approximation error $O(\sigma)$ on the measure.
This error decays but accumulates, and the total error on the physical measure is $O(\sigma/(1-\theta))$, which can be quite large compared to the one-step error $\sigma$.
Hence, if we are interested in the trend between $\Phi_{avg}$ and $\gamma$ for a certain stepsize $\Delta \gamma$ (this is typically known from the practical problem), then the error in the linear response is $O(\sigma/\Delta \gamma (1-\theta))$.
This explains why the random system has linear response whereas the deterministic system might not, since the error is large for $\Delta \gamma$ small.
Together we have the total error $\eps$
\begin{equation} \begin{split} \label{e:lian}
  \eps 
  = O(\theta^W) 
  + O \left(\frac {\sqrt W} {\sigma\sqrt L} \right)
  + O \left(\frac {\sigma} {\Delta \gamma (1-\theta)} \right).
\end{split} \end{equation}

Again, in practice we want the three errors to be roughly equal, which shall prescribe the size of $\sigma$,
\[ \begin{split}
  \sigma = \eps \Delta \gamma (1-\theta),
  \quad \textnormal{so} \quad 
  L 
  = O \left( \frac {\log _\theta \eps} {\sigma^2 \eps^{2}} \right)
  = O \left( \frac {\log _\theta \eps} {\eps^{4}  (1-\theta)^2} \right).
\end{split} \]
Since $1-\theta$ can be small, this cost can be much larger than just $\eps^{-4}$, which is already a high cost.

Finally, we acknowledge that our estimation is very inaccurate, but the point we make is solid, that is, the small noise limit is numerically expensive to achieve.

\subsection{A potential program to unify three linear response formulas}
\hfill\vspace{0.1in}
\label{s:unify}

We sketch a potential program on how to further reduce the cost/error of computing approximate linear response of non-hyperbolic deterministic systems.
As is known, non-hyperbolic systems do not typically have linear responses, so we must mollify, and in this paper we choose to mollify by adding noise in the phase space during each time step.
But as we see in the last subsection, adding a big noise increases the noise error, the third term in \cref{e:lian}; whereas small noise increases the sampling error, the second term in \cref{e:lian}.

The plausible solution is to add a big but local noise, only at locations where the hyperbolicity is bad.
For continuous-time case there seems to be an easy choice: we can let the noise scale be reverse proportional to the flow vector length.
This should at least solve the singular hyperbolic flow cases \cite{Viana2000,Wen2020}, where the bad points coincide with zero velocity.
But for discrete-time case it can be difficult to find a natural criteria which is easy to compute.

The benefit of this program is that, if the singularity set is low-dimensional, then the area where we add big noise is small, and the noise error is small.
On the other hand, we only use the no-propagate formula (the generalized version in \cref{s:generalize}) where the noise is big, so the sampling error is also small.
Where the noise is small, we use the fast-response algorithm, which is efficient regardless of noise scale, but requires hyperbolicity.

This program hinges on the assumption that the singularity set is small or low-dimensional.
It also requires us to invent one more technique, which transfers information from the no-propagate formula to the fast formulas.
For example, the equivariant divergence formula, in its original form, requires information from the infinite past and future \cite{TrsfOprt}.
But we can rerun the proof and restrict the time dependence to finite time; this requires extra information on the interface, such as the derivative of the conditional measure and the divergence of the holonomy map.
These information should and could be provided by the no-propagate formula.

Historically, there are three famous linear response formulas.
The ensemble formula does not work for chaos; 
the operator formula is likely to be cursed by dimensionality;
the random formula is expensive for small noise limit.
The fast response formula combines the ensemble and the operator formula, is given by recursive relations on one orbit, so is neither affected by chaos or high-dimension, but is rigid on hyperbolicity.
Now this paper writes the random formula on one orbit.
Looking into the future, besides trying to test above methods on specific tasks, we should also try to combine all three formulas together into one, which might provide the best approximate linear responses with highest efficiency.

\subsection{Comparison with kernel trick and ensemble backpropagation}
\hfill\vspace{0.1in}
\label{s:kernel}

The kernel trick for computing the derivative of $\Phi_{avg}(\gamma)$ is to compute the derivative of a mollified version, $\Phi_{avg} * \eta$, where $\eta(\cdot)$ is a smooth kernel in the \textit{parameter space}.
So the derivative can be moved to the kernel, and 
\[ \begin{split}
  \delta \Phi_{avg} (\gamma)
  \approx
  \delta (\Phi_{avg} * \eta)
  = \Phi_{avg} * \delta\eta
  = \int \Phi_{avg} (\gamma+\xi) \frac{\delta\eta}\eta (\xi)  \eta(\xi)d\xi
\end{split} \]
Then we sample the convolution by sampling $\xi$ according to $\eta$ and compute averages.
Also note that computing $\Phi_{avg}$ for each $\gamma$ further requires $L$ many samples.

The no-propagate algorithm convolutes in the \textit{phase space} at each time step, which might be more physically meaningful.
Also, our algorithm is more efficient when there are many parameters.
The marginal cost for a new parameter is much lower than the kernel trick, since the kernel trick roughly requires one more $\Phi_{avg}$ for one more parameter.

Then we compare with the ensemble method.
Denote the measure $\rho(\cdot):=\int (\cdot) h dx$, the no-propagate formula is the sum of terms like
\[ \begin{split}
  \rho(\delta f \cdot \frac {dp}p \Phi)
  \approx 
  \rho(\delta f \cdot \frac 1\sigma \Phi)
\end{split} \]
Here $\sigma$ is the scale of the noise.
Assuming there is no gradient vanishing or exploding, the ensemble formula is sum of terms like
$ \rho(\delta f \cdot d\Phi ) $.
The number of samples for the two method are similar when two integrand are of similar size, that is
\[ \begin{split}
  \sigma\approx \frac \Phi {|d\Phi|}\approx 
  \textnormal{ size of support of } \rho =:r
\end{split} \]
if $\Phi$ is not oscillating:
because after centralizing $\Phi$, the average of $\Phi$ is basically the variation of $\Phi$, which basically equals the integration of $d\Phi$ over the support of $\rho$.
But if $\Phi$ is oscillating, then $\frac \Phi {|d\Phi|}$ is typically smaller than $r$, so the no-propagate has similar cost as ensemble for $\sigma<r$.

If we want to use noise case to approximate deterministic case, then we want $\sigma\sim0.3r$, and the no-propagate would require about 10 times more samples than ensemble.
However, no-propagate does not involve backpropagation, so it does not require computing Jacobian or saving a forward orbit, so it is faster than ensemble per orbit.
Overall the cost should be similar; but the cost of the no-propagate algorithm is not affected by chaos.

Summarizing, the cost of the no-propagate and the ensemble algorithm are similar, but the trade-off is a relatively large noise error versus the ability to overcome chaos.

\section*{Acknowledgements}

The author is in great debt to  Caroline Wormell, Wael Bahsoun, and Gary Froyland for very helpful discussions.
This paper is partially supported by the China Postdoctoral Science Foundation 2021TQ0016 and the International Postdoctoral Exchange Fellowship Program YJ20210018.
This work is partially done during the author's postdoc at Peking University and during his visit to Mark Pollicott at the University of Warwick.

\section*{Data availability statement}
The code used in this manuscript is at \url{https://github.com/niangxiu/np}.
There is no other associated data.



\bibliographystyle{abbrv}
{\footnotesize\bibliography{library}}

\end{document}